\documentclass[a4paper]{scrartcl}
\usepackage[a4paper,left=2.4cm,right=2.4cm]{geometry}

\usepackage{fontawesome5}
\usepackage{orcidlink}
\usepackage{graphicx}
\usepackage{amssymb}
\usepackage{amsmath}
\usepackage{amsthm}
\usepackage{amsfonts}
\usepackage[table]{xcolor}
\usepackage{comment}
\usepackage{mathtools}
\usepackage{tikz}
\usetikzlibrary{matrix}
\usepackage{hyperref}
\usepackage[multiple]{footmisc}
\usepackage{subcaption}
\usepackage{enumitem}
\usepackage{stackengine}
\usepackage{tikz}
\usepackage{relsize}
\usepackage{microtype}
\usepackage{caption}
\usepackage{thmtools}
\usepackage{thm-restate}
\usepackage{hyperref}
\usepackage{cleveref}

\usepackage[vlined]{algorithm2e}
\SetArgSty{textnormal}
\DontPrintSemicolon
\SetKwProg{Fn}{}{}{end}
\SetKwFor{ForEach}{for each}{do}{endfch}
\SetKw{Yield}{yield}
\SetKwFor{RepeatTimes}{repeat}{times}{}

\setcapindent{0pt}

\definecolor{darkred}{rgb}{0.545, 0, 0}
\definecolor{darkgreen}{rgb}{0, 0.392, 0}
\definecolor{darkblue}{rgb}{0, 0, 0.545}
\definecolor{darkcyan}{rgb}{0, 0.545, 0.545}
\definecolor{darkmagenta}{rgb}{0.545, 0, 0.545}
\definecolor{darkorange}{rgb}{1, 0.549, 0}
\definecolor{darkviolet}{rgb}{0.58, 0, 0.827}
\definecolor{darkyellow}{rgb}{0.545, 0.545, 0}
\definecolor{gold}{rgb}{1, 0.843, 0}
\definecolor{gold1}{rgb}{0.933, 0.788, 0}
\definecolor{gold2}{rgb}{1, 0.843, 0}
\definecolor{gold3}{rgb}{0.804, 0.678, 0}
\definecolor{gold4}{rgb}{0.545, 0.459, 0}

\newcommand{\walrus}{\mathrel{\coloneq}}

\newcommand{\OEIS}[1]{\textsc{Oeis} #1}
\newcommand{\bigO}[1]{\mathcal{O}(#1)}

\newcommand{\rrr}[1]{\textcolor{darkred}{#1}}
\newcommand{\ooo}[1]{\textcolor{darkorange}{#1}}
\newcommand{\yyy}[1]{\textcolor{gold3}{#1}}
\renewcommand{\ggg}[1]{\textcolor{darkgreen}{#1}}  
\newcommand{\bbb}[1]{\textcolor{darkblue}{#1}}
\newcommand{\vvv}[1]{\textcolor{darkmagenta}{#1}}

\newcommand{\BINARY}[2][]{\mathbf{B}_{#1}(#2)}  
\newcommand{\UNLABELED}[2][]{\mathbf{U}_{#1}(#2)} 
\newcommand{\BINARYzero}[1]{\BINARY[0]{#1}}
\newcommand{\PERMS}[2][]{\mathbf{S}_{#1}(#2)}     
\newcommand{\MATCHINGS}[2][]{\mathcal{M}_{#1}(#2)}
\newcommand{\MATCHINGSMaxSize}[1]{\MATCHINGS[\nu]{G}}

\newcommand{\torusOp}{\otimes}

\newcommand{\wordSet}[2][]{\langle #2 \rangle_{#1}}

\newcommand{\smallTimes}{\raisebox{0.1em}{\scalebox{0.7}{$\times$}}}

\newcommand{\flip}[1]{\overline{#1}}

\newcommand{\wildcard}[1]{\setlength{\fboxsep}{1.5pt}\fbox{#1}}

\newcommand{\pair}[2]{\begin{smallmatrix} #1\\ #2 \end{smallmatrix}}
\newcommand{\pairs}[1]{\big[#1\big]}

\newcommand{\proofcase}[1]{\noindent\textbf{Case #1:}}

    {\begin{enumerate}[nolistsep,noitemsep]}%
    {\end{enumerate}}
    
\newenvironment{tightitemize}%
    {\begin{itemize}[nolistsep,noitemsep]}%
    {\end{itemize}}

\declaretheorem[name=Theorem]{theorem}
\declaretheorem[name=Lemma, sibling=theorem]{lemma}

\declaretheorem[name=Definition, sibling=theorem, style=definition]{definition}
\declaretheorem[name=Remark, sibling=theorem, style=remark]{remark}

\newcommand{\Kn}[1]{K_{#1}}

\title{Shorthand Universal Tori for Permutations} 
\title{Shorthand Universal Tori for Permutations: Existence, Symmetry, and Generation of Twori}

\setkomafont{author}{\normalsize}

\ExplSyntaxOn
\NewDocumentCommand \authormails { m }
  {
    \seq_set_split:Nnn \l_tmpa_seq { \and } {#1}
    \seq_map_inline:Nn \l_tmpa_seq
      { \, \href{mailto:##1}{\footnotesize\faEnvelope[regular]} }
  }
\ExplSyntaxOff

\makeatletter
\gdef\@author{}
\newcommand*\ifargempty[1]{\if\relax\detokenize{#1}\relax\expandafter\@firstoftwo\else\expandafter\@secondoftwo\fi}
\renewcommand*\author[5]{%
  \ifx\@author\@empty\else\g@addto@macro\@author{\and}\fi
  \g@addto@macro\@author{%
    #1%
    \ifargempty{#3}{}{\,\,\authormails{#3}}%
    \ifargempty{#4}{}{\,\orcidlink{#4}}%
    \ifargempty{#5}{}{\thanks{#5}}%
    \\{\footnotesize\normalfont\def\and{\\}%
      \begin{tabular}[t]{@{}c@{}}#2\end{tabular}}%
  }%
}
\makeatother

\author{Tim Gerlach}{Universität Hamburg, Germany}{tim.gerlach@uni-hamburg.de}{https://orcid.org/0009-0004-6294-9235}{}
\author{Elizabeth Hartung}{Massachusetts College of Liberal Arts, United States}{e.hartung@mcla.edu}{https://orcid.org/0000-0002-4041-1862}{}
\author{Pia Herkenrath}{Universit\"at Kassel, Germany}{pia.herkenrath@mathematik.uni-kassel.de}{https://orcid.org/0009-0001-9504-297X}{}
\author{Joe Sawada}{University of Guelph, Canada}{jsawada@uoguelph.ca}{https://orcid.org/0000-0001-7364-2993}{}
\author{Aaron Williams}{Lakehead University, Canada\and Williams College, United States}{aaron.williams@lakeheadu.ca\and aaron.williams@williams.edu}{https://orcid.org/0000-0001-6816-4368}{}

\begin{document}

\maketitle

\begin{abstract}
A de Bruijn sequence packs all $n$-bit binary words into a cycle of length $2^n$.
A de Bruijn torus is the two-dimensional analogue in which each word appears exactly once in a rectangular window.
Here we consider the natural analogue for permutations using their shorthand representation (i.e., each permutation's final redundant value is omitted from the window).
We show that these tori exist when $n=2m+1$ is odd and the torus and windows have two rows (i.e., the torus is a ``tworus'').
These twori can be constructed with a high degree of symmetry.
More specifically, there are twori that can be partitioned into $2^{m-1}$ matching blocks where each block contains the same sequence of unordered columns.
Furthermore, given one such block we can generate each successive column of a tworus in amortized $\bigO{1}$-time.
We also prove non-existence results for certain sizes of tori and provide algorithms for constructing multiversal cycles (perfect necklaces) of unlabeled binary words. 
\end{abstract}

\section{Introduction}
\label{sec:intro}

Let $\BINARY{n}$ be the set of $n$-bit binary words.
A \emph{de Bruijn sequence} of $\BINARY{n}$ is a circular word of length $2^n$ where each word in $\BINARY{n}$ appears exactly once as a substring.
De Bruijn sequences are fundamental objects in Computer Science \cite{graham1994concrete,knuthseries}.
In two dimensions, a \emph{binary de Bruijn torus} of $\BINARY{n}$ is an $R$-by-$C$ toroidal grid with $R \cdot C = 2^n$ where each word in $\BINARY{n}$ appears exactly once in row major order in a rectangular $r$-by-$c$ window with $r \cdot c = n$.
Figure \ref{fig:torusExamples_binary} has an example with $(R,C;r,c) = (8,32;2,4)$ for $n=8$.
The concept is also known as a \emph{binary array} \cite{ma1984note}, \emph{de Bruijn array} \cite{fan85debruijn}, \emph{perfect map}~\cite{paterson1994perfect}, and \emph{pseudo-random array} (with no all-zero window) \cite{macwilliams1976pseudo}.
Applications include efficient position encoding and decoding \cite{shiu1997decoding,horan2016locating,etzion2024bruijn}.

\begin{figure}[h]
    \centering
    \begin{subfigure}{0.4925\columnwidth}
        \centering
        \includegraphics[width=\textwidth,page=3]{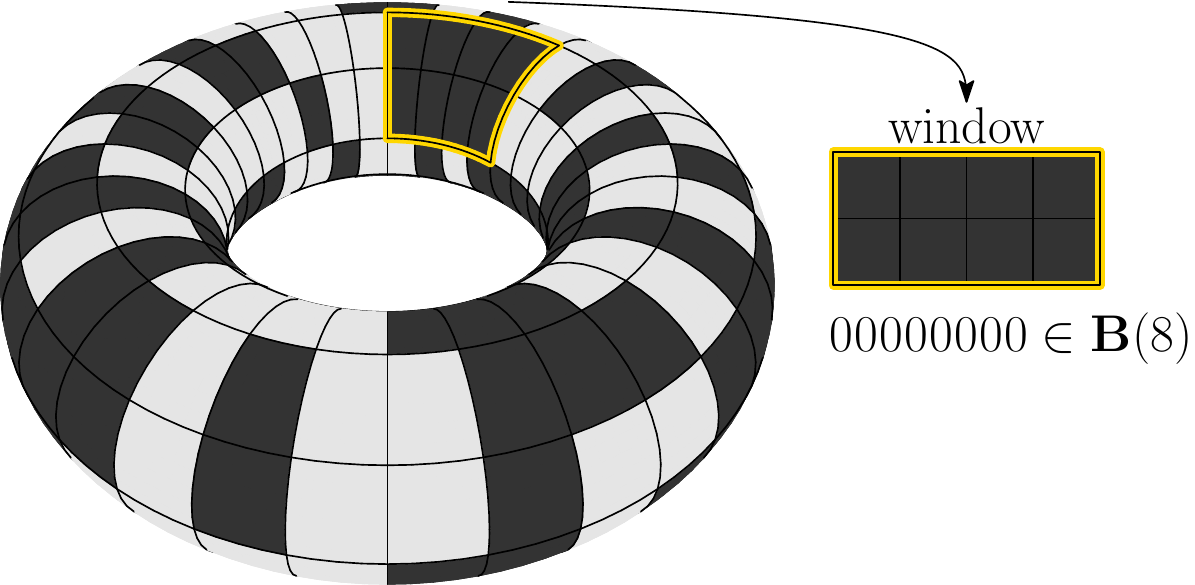}
        \caption{A binary de Bruijn torus with shape $(R,C;r,c)=(8,32;2,4)$.
        The $8$-by-$32$ toroidal grid has a unique $2$-by-$4$ window for each $n$-bit word with $n = r \cdot c = 8$.}
        \label{fig:torusExamples_binary}        
    \end{subfigure}
    \hfill
    \begin{subfigure}{0.4925\columnwidth}
        \centering
        \includegraphics[height=1.35in,page=5]{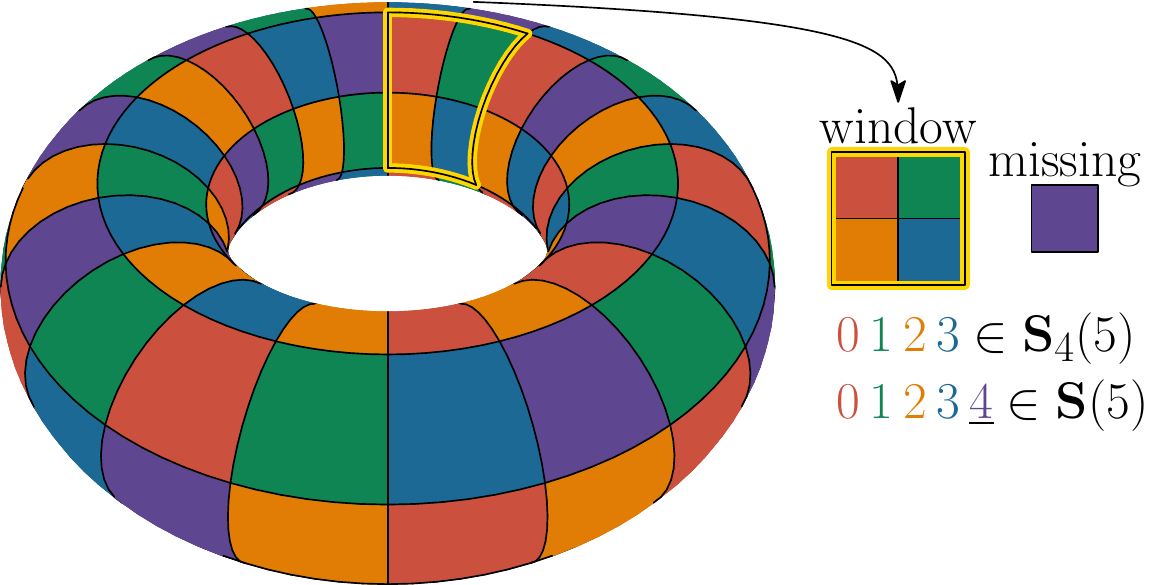}
        \caption{A shorthand universal torus of permutations for $(6,20;2,2)$.
        The $6$-by-$20$ torus has a unique $2$-by-$2$ window for each permutation with $n = r \cdot c + 1 = 5$.}
        \label{fig:torusExamples_shorthand}        
    \end{subfigure}
    \caption{Universal tori in (a) and (b) with their lexicographically smallest windows highlighted.
    The $2$-by-$3$ windows of (a) encode distinct binary words in $\BINARY{n}$ for $n=8$ in row major order while the $2$-by-$2$ windows of (b) encode distinct $(n{-}1)$-permutations $\PERMS[n-1]{n}$ for $n=5$.
    Each $\PERMS[n-1]{n}$ is shorthand for a permutation in $\PERMS{n}$, so (b) is a shorthand universal torus of permutations.
    }
    \label{fig:torusExamples}
\end{figure}

The term \emph{universal cycle} is used for de Bruijn sequences of other objects \cite{chung1992universal} and here we consider permutations.
Let $\PERMS{n}$ be the set of permutations of $\{0, 1,\ldots,n-1\}$ in one-line notation and $\PERMS[k]{n}$ be the set of $k$-permutations (e.g., $\PERMS[2]{3} = \{01,02,10,12,20,21\}$).
When $n > 2$ it is not possible to create a circular word of length $n!$ in which every word in $\PERMS{n}$ appears exactly once as a substring.
For example, if $0 \, 1 \, 2 \, \wildcard{\phantom{0}} \, \wildcard{\phantom{1}} \, \wildcard{\phantom{2}}$ were a sequence for $n=3$, then the first $\wildcard{\phantom{0}}$ is $\wildcard{0}$ as $1 \, 2 \, \wildcard{\phantom{0}}$ is a window; continuing this way gives $012012$ which does not suffice.
While permutations cannot be directly stored within a universal cycle, they can be stored indirectly.
For example, $012032$ stores the permutations using relative order (e.g., $203$ encodes $102$) and in general $n+1$ distinct symbols are sufficient \cite{johnson2009universal}.
Other approaches to containing all permutations have been considered \cite{engen2020containing} including shorthand described below.

A \emph{shorthand universal cycle} of $\PERMS{n}$ omits the final (redundant) value in each permutation.
More precisely, it is a circular word of length $n!$ in which every word in $\PERMS[n-1]{n}$ appears exactly once as a length $n-1$ substring.
Each word in $\PERMS[n-1]{n}$ encodes a unique permutation by suffixing its missing symbol.
For example, $01\underline{20}21$ is suitable for $n=3$ and its underlined substring $20 \in \PERMS[2]{3}$ encodes $201 \in \PERMS{3}$.
These cycles exist for all $n$ \cite{jackson1993universal} and can be generated efficiently \cite{ruskey2010explicit,holroyd2012shorthand,chang2026efficient}.
Here we introduce a natural two-dimensional analogue.

A \emph{shorthand universal torus} of $\PERMS{n}$ is an $R$-by-$C$ toroidal grid with $R \cdot C = n!$ in which every word in $\PERMS[n-1]{n}$ appears exactly once in row major order within a rectangular $r$-by-$c$ window with $r\cdot c=n-1$.
Figure \ref{fig:torusExamples_shorthand} has an example for $(R,C;r,c) = (6,20;2,2)$ and $n=5$.

An initial question is for which values of $(R,C;r,c)$ these tori exist.
We view this as a \emph{Gray code} question \cite{savage1997survey,mutze2023combinatorial} as the tori are patterns of a combinatorial object $\PERMS{n}$ with certain restrictions.
A second question is if these tori can be constructed efficiently.
We view this as a \emph{combinatorial generation} question \cite{ruskey2003combinatorial} as the goal is the efficient generation of $\PERMS{n}$.

A recent trend in Gray codes is the creation of orders with high degrees of symmetry \cite{gregor2024hamilton}.
A notable example involves the \emph{middle levels} (i.e., the subset of $\BINARY{2m+1}$ with words having $m$ or $m+1$ copies of $1$).
A Gray code exists with $2n+1$ blocks with rotational symmetry \cite{merino2021combinatorial,merino2022combinatorial} which strengthens the famous middle levels theorem \cite{mutze2016proof,mutze2024book} and answers a challenge posed by Knuth \cite{knuthseries}.
Other notions of symmetry have been investigated \cite{sawada2019solving,rytter2021syntactic}.
Combinatorial generation motivates this trend as repeated patterns are easier to generate.

We prove that shorthand universal tori of $\PERMS{n}$ exist when $n=2m+1$ and both the torus and windows have two rows.
Moreover, our proof is constructive and highly symmetric.
We say that a pair of columns \emph{match} if they contain the same unordered set of values.
More broadly, two blocks of length $k$ (i.e., two runs of $k$ contiguous columns) \emph{match} if their columns pairwise match.
For example, $\pairs{\pair{0}{2} \pair{1}{3} \pair{0}{4}}$ and $\pairs{\pair{2}{0} \pair{1}{3} \pair{4}{0}}$ are two matching blocks.

\begin{restatable}{theorem}{tworusExists}\label{thm:tworus}
    There exist shorthand universal tori of permutations $\PERMS{n}$ of size $2$-by-$\frac{n!}{2}$ with $2$-by-$m$ windows for all odd $n=2m+1$.
    Moreover, they partition into $2^{m-1}$ matching blocks.
\end{restatable}

\begin{figure}[t]
    \centering
    \begin{subfigure}{0.48\columnwidth}
        \centering
        \includegraphics[width=\textwidth,page=2]{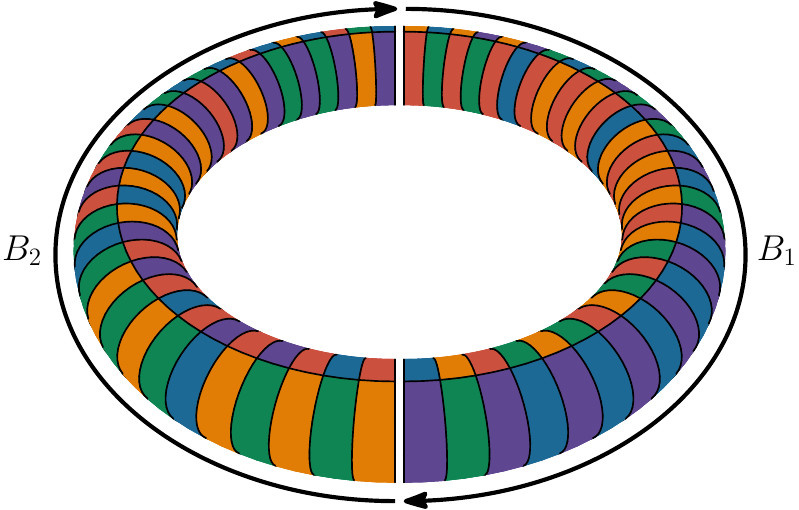}
        \caption{A shorthand universal tworus for $\PERMS{5}$.}
        \label{fig:blocks_blocks}
    \end{subfigure}
    \hfill
    \begin{subfigure}{0.51\columnwidth}
        \begin{subfigure}{\columnwidth}
            \centering
            \includegraphics[scale=0.475,page=3]{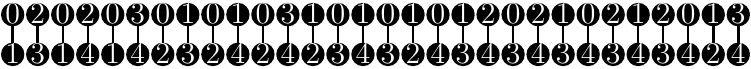}
            \caption{A universal cycle of $\MATCHINGS[2]{\Kn{5}}$.}
            \label{fig:blocks_matching}
        \end{subfigure}
        \begin{subfigure}{\columnwidth}
            \smallskip
            \centering
            \includegraphics[scale=0.475,page=4]{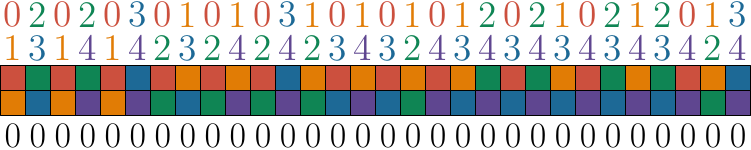}
            \caption{Block $B_1$'s columns order (b)'s edges using $00\cdots$.}
            \label{fig:blocks_B1}
        \end{subfigure}
        \begin{subfigure}{\columnwidth}
            \smallskip
            \centering
            \includegraphics[scale=0.475,page=5]{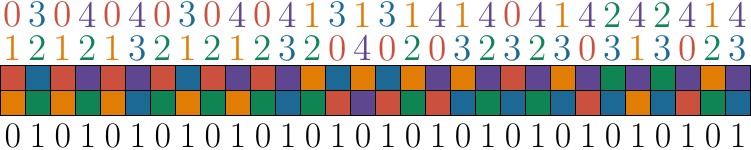}
            \caption{Block $B_2$'s columns order (b)'s edges using $01\cdots$.}
            \label{fig:blocks_B2}
        \end{subfigure}
    \end{subfigure}
    \caption{A tworus for $n=2m+1=5$ in (a).
    It is a shorthand universal torus of dimensions $(R,C;r,c) = (2,60;2,2)$ for $\PERMS{5}$, or equivalently, a universal torus for $\PERMS[4]{5}$.
    It contains $2^{m-1}=2$ matching blocks of length $\frac{n!}{2^m} = 30$.
    The unordered columns of each block form a universal cycle of ordered maximum matchings in the complete graph $\Kn{5}$ in (b).
    The columns are ordered with bits from a multiversal cycle of unlabeled binary words $\BINARY[0]{m} = \{00,01\}$ as in (c)--(d).
    }
    \label{fig:blocks}
\end{figure}


Throughout the paper we let $[n] = \{0,1,\ldots,n{-}1\}$ and we index our words and sequences with $0$-based indices that are taken modulo their length.
For brevity (and levity) we use the term \emph{tworus} to refer to a torus with $R=2$ rows using windows with $r=2$ rows.

Section \ref{sec:n5} provides existence and non-existence results obtained via exhaustive computation.
Section \ref{sec:uniMatchings} provides new results on universal cycles for ordered matchings and Sections \ref{sec:binary}--\ref{sec:unlabeled} provides multiversal cycles of (unlabeled) binary words.
Section \ref{sec:tworiPerms} constructs our twori using a universal cycle of ordered matchings as the repeating block structure, and by orienting the pairs in successive blocks using a multiversal cycle.
Section \ref{sec:alg} provides algorithms including $\bigO{1}$-time per column generation of a twori given a suitable universal cycle.

\section{Exhaustive Search with 2-by-2 Windows}
\label{sec:n5}

The smallest possible examples of shorthand universal tori of $\PERMS{n}$ with non-linear windows are for $n=5$ with $2$-by-$2$ windows.
Given these choices a suitable torus must have size $R$-by-$C$ with $R \cdot C = 120$.
It is also clear that $R, C > 1$ since otherwise the windows contain repeated values.
Finally, without loss of generality we can also assume that $R \leq C$.

Two of the authors independently wrote backtracking programs to exhaustively search for existence in each of the seven possible cases: $R \in \{2,3,4,5,6,8,10\}$.
A window 
$[\begin{smallmatrix} a & c \\ b & d \end{smallmatrix}]$
is \emph{smaller} than 
$[\begin{smallmatrix} w & y \\ x & z \end{smallmatrix}]$
if $abcd < wxyz$ in lexicographic order.  
The searches find the shorthand universal torus (if one exists) with the smallest windows as they are found top-to-bottom, then left-to-right.
The examples that exist are in Figure \ref{fig:n5} and the existence results are summarized in Table \ref{tab:exist5}.
We prove two of the non-existence results by hand in Section \ref{sec:nexists}.
\begin{restatable}{theorem}{tThreeNonexistence}\label{thm:t3_nonexistence}
    There is no $3$-by-$40$ shorthand universal torus for $\PERMS{5}$ with $2$-by-$2$-windows.
\end{restatable}
\begin{restatable}{theorem}{tFiveNonexistence}\label{thm:t5_nonexistence}
    There is no $5$-by-$24$ shorthand universal torus for $\PERMS{5}$ with $2$-by-$2$-windows.
\end{restatable}

\begin{figure}[ht]
    \centering
    \begin{subfigure}{0.75\columnwidth}
        \begin{subfigure}{1.0\columnwidth}
            \centering
            \includegraphics[scale=0.6,page=2]{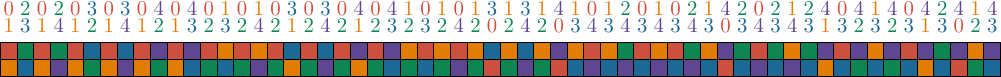}
            \caption{$2$-by-$60$}
            \label{fig:n5_2_60}
        \end{subfigure}
        \begin{subfigure}{0.59\columnwidth}
            \centering
            \includegraphics[scale=0.6,page=3]{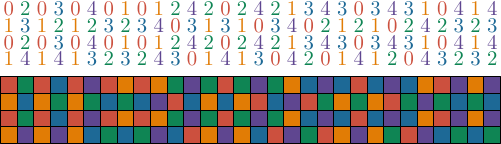}
            \caption{$4$-by-$30$}
            \label{fig:n5_4_30}
        \end{subfigure}
        \begin{subfigure}{0.39\columnwidth}
            \centering
            \includegraphics[scale=0.6,page=4]{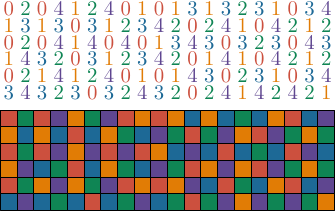}
            \caption{$6$-by-$20$}
            \label{fig:n5_6_20}
        \end{subfigure}
    \end{subfigure}
    \hfill
    \begin{subfigure}{0.24\columnwidth}
        \centering
        \includegraphics[scale=0.6,page=5]{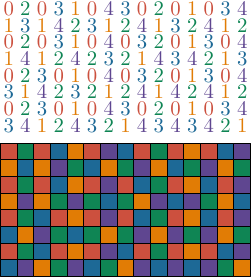}
        \caption{$8$-by-$15$}
        \label{fig:n5_8_15}
    \end{subfigure}
    \caption{The lexicographically least shorthand universal tori for $n=5$ and $2$-by-$2$ windows of the sizes in Table \ref{tab:exist5}.
    Note that our tworus in Figure \ref{fig:blocks_blocks} is not the same as the one here in (a).}
    \label{fig:n5}
    \medskip
    \small
    \begin{tabular}{|@{\;\;}c@{\;\;}|@{\;\;}c@{\;\;}|@{\;\;}c@{\;\;}|@{\;\;}c@{\;\;}|@{\;\;}c@{\;\;}|@{\;\;}c@{\;\;}|@{\;\;}c@{\;\;}|@{\;\;}c@{\;\;}|} \hline
    $1 \smallTimes 120$ & $2 \smallTimes 60$ & $3 \smallTimes 40$ & $4 \smallTimes 30$ & $5 \smallTimes 24$ & $6 \smallTimes 20$ & $8 \smallTimes 15$ & $10 \smallTimes 12$ \\ \hline
    no & yes & no & yes & no & yes & yes & no \\ \hline
    \end{tabular}
    \captionof{table}{Existence of shorthand universal tori of size $R$-by-$C$ for the permutations $\PERMS{5}$ with $2$-by-$2$ windows.
    These results were obtained by exhaustive computation using backtracking.
    }
    \label{tab:exist5}
\end{figure}

Table \ref{tab:exist5} suggests that existence for shorthand universal tori of permutations is more complicated than for binary and $q$-ary de Bruijn tori.
During the \emph{Generalizations of de Bruijn cycles and Gray Codes} workshop in 2004, Ron Graham asked if $q$-ary de Bruijn tori can be constructed for $q>1$ given (i) $RC = q^n$, (ii) $R > r$, and (iii) $C > c$; see Question 1 of Problem 480 in \cite{jackson2009research}.
Note that $R \geq r$ is necessary in this context, and $R>r$ is necessary unless the torus is a horizontal sequence (i.e., $R=r=1$) as otherwise the all-zero window is repeated.
Similarly, (iii) is necessary unless $C=c=1$.
Thus, the question is if the natural necessary conditions for a $q$-ary de Bruijn torus are sufficient.
Earlier papers conjectured similar conditions and verified them for special cases including when $q=2$ \cite{paterson1994perfect} or $q$ is a prime power \cite{paterson1996new}.
Note (ii) and (iii) are not necessary for shorthand universal tori of permutations.
Indeed our primary focus is on the tworus case of $R=r=2$.
Our Theorem \ref{thm:tworus} also proves one case of Graham's question for $rc$-permutations; see Question 2 of Problem 480 in \cite{jackson2009research}.

\section{Universal Cycles for Ordered Matchings}
\label{sec:uniMatchings}


Here we consider the pattern of the blocks in our tworus construction. 
Each column in the pattern is an unordered pair.
The pairs in each column are ordered in different ways in different blocks, as determined in Section \ref{sec:tworiPerms}, to create a tworus. 


We formalize the pattern as a universal cycle of ordered matchings.
We prove that they exist using the approach from de Bruijn's original paper \cite{de1946combinatorial}.
That is, we argue that they are in one-to-one correspondence with the Eulerian circuits in an Eulerian transition graph%
\footnote{Some authors define the transition graph as the line graph of our notion of a transition graph.}.

An \emph{ordered matching} in an undirected graph $G = (V, E)$ is an ordered list of distinct undirected edges $L = [e_1, e_2, \ldots, e_k]$ in which $M  = \{e_1, e_2, \ldots, e_k\}$ is a matching.
Let $\MATCHINGS[k]{G}$ be $G$'s set of ordered matchings of size $k$.
For example, $\MATCHINGS[1]{G} = \{[e] \mid e \in E\}$ for $k=1$.

A \emph{universal cycle of $\MATCHINGS[k]{G}$} is a cyclic sequence containing all $L \in \MATCHINGS[k]{G}$ exactly once.
Each edge $e = \{u, v\} \in E$ appears as an unordered indivisible unit in these universal cycles.
An example $\pairs{\pair{0}{1}, \pair{2}{3}, \pair{0}{1}, \pair{2}{4}, \ldots}$ for $n=5$ previously appeared in Figure \ref{fig:blocks_matching} where each $\pair{u}{v}$ denotes $\{u, v\}$.
The empty sequence $\epsilon$ suffices when $\MATCHINGS[k]{G} = \emptyset$.


We are primarily interested in universal cycles for the near-perfect matchings of size $m$ in the complete graph $\Kn{n}$ when $n=2m+1$;
see Figure \ref{fig:transitionMatching_maximum5} for the transition graph when $n=5$.
However, we prove a more general result in Theorem \ref{thm:uniMatchings}.
The most challenging part of the proof is connectivity, and our substitution argument is illustrated in Figure \ref{fig:substitution}.

\begin{figure}
    \centering
    \begin{subfigure}{0.49\columnwidth}
        \begin{subfigure}{1.0\columnwidth}
            \centering
            \includegraphics[scale=1.2,page=14]{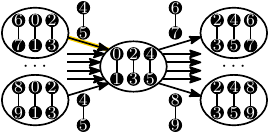}
            \caption{Neighborhood of $L = \pairs{\pair{0}{1}, \pair{2}{3}, \pair{4}{5}}$ in the transition graph for $\MATCHINGS[k]{\Kn{n}}$ with $k=4$ and $n=10$.
            The highlighted edge encodes the ordered matching $\pairs{\pair{6}{7}, \pair{0}{1}, \pair{2}{3}, \pair{4}{5}}$ of size $k=4$ in $\Kn{10}$.
            }
            \label{fig:transitionMatching_neighborhood}
        \end{subfigure}
        \begin{subfigure}{1.0\columnwidth}
            \centering
            \includegraphics[scale=0.95,page=10]{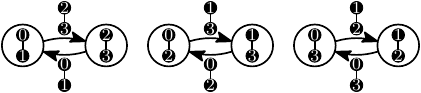}
            \caption{
            The transition graph of $\MATCHINGS[k]{\Kn{n}}$ for $n=4$ and $k=2$ is disconnected.
            }
            \label{fig:transitionMatching_disconnected}
        \end{subfigure}
    \end{subfigure}
    \hfill
    \begin{subfigure}{0.49\columnwidth} 
        \centering
        \includegraphics[scale=0.925,page=11]{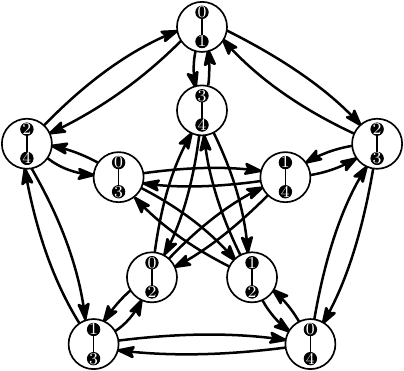}
        \caption{
        Transition graph of $\MATCHINGS[2]{\Kn{5}}$. 
        The edge labels (omitted) on its Eulerian circuits are universal cycles for ordered maximum matchings in~$\Kn{5}$ (e.g. Fig.~\ref{fig:blocks_matching}).
        }
        \label{fig:transitionMatching_maximum5}
    \end{subfigure}
    \caption{Transition graphs for ordered matchings in a complete graph $\MATCHINGS[k]{\Kn{n}}$.
    The node set of the transition graph is $\MATCHINGS[k-1]{\Kn{n}}$ and each node has in-degree and out-degree $\binom{n - 2(k-1)}{2}$ as seen in (a).
    The transition graph is disconnected if and only if $n = 2k$ (i.e., the ordered matchings are perfect) as seen in (b) for $n=4$.
    The transition graph is connected in (c) so the edge labels on its Eulerian circuits are universal cycles for $\MATCHINGS[k]{\Kn{n}}$.
    }
    \label{fig:transitionMatching}
\end{figure}

\begin{theorem} \label{thm:uniMatchings}
A universal cycle for $\MATCHINGS[k]{\Kn{n}}$ exists except when $n \geq 4$ is even and $k = \frac{n}{2}$.
In other words, universal cycles for the ordered matchings of size $k$ in the complete graph exist except when they are perfect matchings with more than one edge.
\end{theorem}

\begin{proof}
Let $G = \Kn{n} = (V, E)$ and consider the associated transition graph whose node set is $\MATCHINGS[k-1]{G}$. We first show that the transition graph is balanced.
Let $L = [e_1, e_2, \ldots, e_{k-1}] \in \MATCHINGS[k-1]{G}$ be an ordered matching and $M = \{e_1, e_2, \ldots, e_{k-1}\}$ be the associated matching.
Note that $[e_1, e_2, \ldots, e_{k-1}, e] \in \MATCHINGS[k]{G}$ if and only if $[e, e_1, e_2, \ldots, e_{k-1}] \in \MATCHINGS[k]{G}$ if and only if $e$'s endpoints are unmatched by $M$.
Thus the out-degree and in-degree of $L$ in the transition graph are both $\binom{n - 2(k-1)}{2}$.
Since the transition graph is balanced, a universal cycle exists if and only if it is (weakly) connected.
(Note that a balanced directed graph is strongly connected if and only if it is weakly connected.)

For the negative case, suppose that $n$ is even and $k = \frac{n}{2}$. 
Then each $L \in \MATCHINGS[k-1]{G}$ has in-degree and out-degree $\binom{n - 2(k-1)}{2} = \binom{2}{2} = 1$ in the transition graph.
Thus, the transition graph partitions into directed cycles  each of length $\frac{n}{2}$.
When $n \geq 4$ this partition includes more than one cycle, so the transition graph is disconnected and no universal cycle exists.

Otherwise, $k < \frac{n}{2}$.
We argue that there is a path from $L = [e_1, \ldots, e_i, \ldots, e_{k-1}] \in \MATCHINGS[k-1]{G}$ to $L' = [e_1, \ldots, e_i', \ldots, e_{k-1}] \in \MATCHINGS[k-1]{G}$ where $e_i = \{u,v\}$ and $e_i' = \{u,v'\}$ differ in a single vertex.
In other words, we can substitute vertex $v$ with vertex $v'$ in the ordered matching.
Note that this necessarily implies that $v'$ is unmatched in $L$ and $v$ is unmatched in $L'$.
Repeated substitutions prove that the transition graph is connected.

By our choice of $k$ there are at least three vertices that are unmatched by $L$ including $v'$.
Let $y \neq z$ be two such vertices distinct from $v'$, and let $e_k = \{y,z\}$.
The path is explicitly constructed below where each $\rightarrow$ denotes a single arc in the transition graph and $\rightarrow^*$ denotes a sequence of arcs.
We also keep track of the unmatched (pairs of) vertices in the left column. 
\begin{align*}
v',e_k  &&& [e_1, e_2, \ldots, e_i, \ldots, e_{k-1}] \rightarrow \\ 
v',e_1     &&& [e_2, \ldots, e_i, \ldots, e_{k-1}, e_k] \rightarrow^{*} \\
v',e_{i-1} &&& [e_i, \ldots, e_{k-1}, e_k, e_1, \ldots, e_{i-2}] \rightarrow \\
v',e_i  &&& [e_{i+1}, \ldots, e_{k-1}, e_k, e_1, \ldots, e_{i-1}] \rightarrow \\
v,e_{i+1} &&& [e_{i+2}, \ldots, e_{k-1}, e_k, e_1, \ldots, e_{i-1}, e_i'] \rightarrow^* \\
v,e_k  &&& [e_1, \ldots, e_{i-1}, e_i', e_{i+1}, \ldots, e_{k-1}] \qedhere
\end{align*}%
\end{proof}

\begin{figure}
    \centering
    \includegraphics[scale=1.1,page=13]{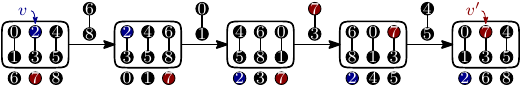}
    \caption{A substitution in the transition graph for $\MATCHINGS[4]{\Kn{9}}$.
    Note that the path substitutes vertex $\bbb{v = 2}$ with $\rrr{v' = 7}$.
    For convenience the unmatched vertices are shown below each node.
    }
    \label{fig:substitution}
\end{figure}

In our tworus construction we orient each edge in these universal cycles.
The following remark ensures that these orientations can be specified using words of length $m$.
For example, in Figures \ref{fig:blocks_B1}--\ref{fig:blocks_B2} the edges are oriented with repeated copies of $00$ and $01$, respectively.



\begin{remark} \label{rem:div}
The number of ordered maximum or near perfect matchings in the complete graph $\Kn{n}$ for odd $n = 2m+1$ is $|\MATCHINGS[m]{\Kn{n}}| = \frac{n!}{2^m}$ (\OEIS{A007019} \cite{oeisMain}) and is divisible by $m$.
\end{remark}

\section{Binary Multiversal Cycles}
\label{sec:binary}

Here we consider variations of binary de Bruijn sequences that contain each $n$-bit word $r$ times instead of once.
In particular, our results focus on $r = kn$ (i.e., the repetition is a multiple of the word length).
Figure \ref{fig:multiversalBinary3} has examples for $n=3$ and multiplicity $r=3,6$.
We refer to these sequences as \emph{multiversal cycles} in reference to the concept that was previously applied to shorthand permutations \cite{ruskey2010explicit}.
Multiversal cycles of binary and $q$-ary words have also been studied using the terms \emph{perfect necklaces} \cite{alvarez2016perfect,becher2024lyndon,fillmore2026existence} and \emph{cyclic multi de Bruijn sequences} \cite{tesler2017multi} albeit with different parameter names.

In Theorem \ref{thm:multiversalBinary} we prove that binary multiversal cycles can be constructed by a lexicographic concatenation reminiscent of the well-known FKM algorithm \cite{fredricksen1977lexicographic,fredricksen1978necklaces} or necklace-prefix algorithm \cite{dragon2016grandmama,dragon2018constructing} and its many generalizations \cite{sawada2016generalizing,dimuro2019classifying}.
This result was previously known \cite{alvarez2016perfect, fillmore2026existence}, but we include our own proof, which was discovered independently, as it forms a basis for our new results on multiversal cycles for unlabeled binary words in Section~\ref{sec:unlabeled}.

Section \ref{sec:binary_multiversal} defines multiversal cycles, then two lemmas in Section \ref{sec:binary_lemmas} lead to our result in Section \ref{sec:binary_construction}.
Unlabeled versions of these cycles are discussed in Section \ref{sec:unlabeled} and are used to orient the columns in our twori.


\begin{figure}
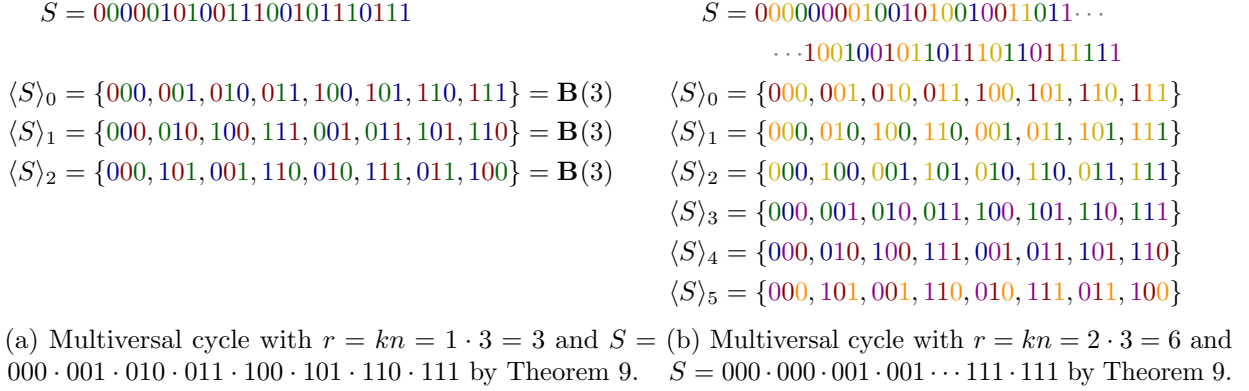

    \begin{subfigure}[t]{0.53\textwidth}
        \small\raggedright
        $\begin{aligned}
            S &= \rrr{0}\ggg{0}\bbb{0}\rrr{0}\ggg{0}\bbb{1}\rrr{0}\ggg{1}\bbb{0}\rrr{0}\ggg{1}\bbb{1}\rrr{1}\ggg{0}\bbb{0}\rrr{1}\ggg{0}\bbb{1}\rrr{1}\ggg{1}\bbb{0}\rrr{1}\ggg{1}\bbb{1} \\
            \\
            \wordSet[0]{S} &= \{ \rrr{0}\ggg{0}\bbb{0}, \rrr{0}\ggg{0}\bbb{1}, \rrr{0}\ggg{1}\bbb{0}, \rrr{0}\ggg{1}\bbb{1}, \rrr{1}\ggg{0}\bbb{0}, \rrr{1}\ggg{0}\bbb{1}, \rrr{1}\ggg{1}\bbb{0}, \rrr{1}\ggg{1}\bbb{1} \} = \BINARY{3} \\
            \wordSet[1]{S} &= \{ \ggg{0}\bbb{0}\rrr{0}, \ggg{0}\bbb{1}\rrr{0}, \ggg{1}\bbb{0}\rrr{0}, \ggg{1}\bbb{1}\rrr{1}, \ggg{0}\bbb{0}\rrr{1}, \ggg{0}\bbb{1}\rrr{1}, \ggg{1}\bbb{0}\rrr{1}, \ggg{1}\bbb{1}\rrr{0} \} = \BINARY{3} \\
            \wordSet[2]{S} &= \{ \bbb{0}\rrr{0}\ggg{0}, \bbb{1}\rrr{0}\ggg{1}, \bbb{0}\rrr{0}\ggg{1}, \bbb{1}\rrr{1}\ggg{0}, \bbb{0}\rrr{1}\ggg{0}, \bbb{1}\rrr{1}\ggg{1}, \bbb{0}\rrr{1}\ggg{1}, \bbb{1}\rrr{0}\ggg{0} \} = \BINARY{3} \\
            \\
            \\
            \\
        \end{aligned}$
        \caption{Multiversal cycle with $r = kn = 1 \cdot 3 = 3$ and
        $S = 000 \cdot 001 \cdot 010 \cdot 011 \cdot 100 \cdot 101 \cdot 110 \cdot 111$ by Theorem~\ref{thm:multiversalBinary}.
        }
        \label{fig:multiversalBinary3_r3}
    \end{subfigure}
    \hfill
    \begin{subfigure}[t]{0.46\textwidth} 
        \small\raggedright
        $\begin{aligned}
            S &= 
            \rrr{0}\ooo{0}\yyy{0}\ggg{0}\bbb{0}\vvv{0}\rrr{0}\ooo{0}\yyy{1}\ggg{0}\bbb{0}\vvv{1}\rrr{0}\ooo{1}
            \yyy{0}\ggg{0}\bbb{1}\vvv{0}\rrr{0}\ooo{1}\yyy{1}\ggg{0}\bbb{1}\vvv{1}\textcolor{gray}{\cdots}\\
            &\qquad\textcolor{gray}{\cdots}
            \rrr{1}\ooo{0}\yyy{0}\ggg{1}\bbb{0}\vvv{0}\rrr{1}\ooo{0}\yyy{1}\ggg{1}\bbb{0}\vvv{1}\rrr{1}
            \ooo{1}\yyy{0}\ggg{1}\bbb{1}\vvv{0}\rrr{1}\ooo{1}\yyy{1}\ggg{1}\bbb{1}\vvv{1} \\
            \wordSet[0]{S} &= \{ \rrr{0}\ooo{0}\yyy{0}, \rrr{0}\ooo{0}\yyy{1}, \rrr{0}\ooo{1}\yyy{0}, \rrr{0}\ooo{1}\yyy{1}, \rrr{1}\ooo{0}\yyy{0}, \rrr{1}\ooo{0}\yyy{1}, \rrr{1}\ooo{1}\yyy{0}, \rrr{1}\ooo{1}\yyy{1} \} \\
            \wordSet[1]{S} &= \{ \ooo{0}\yyy{0}\ggg{0}, \ooo{0}\yyy{1}\ggg{0}, \ooo{1}\yyy{0}\ggg{0}, \ooo{1}\yyy{1}\ggg{0}, \ooo{0}\yyy{0}\ggg{1}, \ooo{0}\yyy{1}\ggg{1}, \ooo{1}\yyy{0}\ggg{1}, \ooo{1}\yyy{1}\ggg{1} \} \\
            \wordSet[2]{S} &= \{ \yyy{0}\ggg{0}\bbb{0}, \yyy{1}\ggg{0}\bbb{0}, \yyy{0}\ggg{0}\bbb{1}, \yyy{1}\ggg{0}\bbb{1}, \yyy{0}\ggg{1}\bbb{0}, \yyy{1}\ggg{1}\bbb{0}, \yyy{0}\ggg{1}\bbb{1}, \yyy{1}\ggg{1}\bbb{1} \} \\
            \wordSet[3]{S} &= \{ \ggg{0}\bbb{0}\vvv{0}, \ggg{0}\bbb{0}\vvv{1}, \ggg{0}\bbb{1}\vvv{0}, \ggg{0}\bbb{1}\vvv{1}, \ggg{1}\bbb{0}\vvv{0}, \ggg{1}\bbb{0}\vvv{1}, \ggg{1}\bbb{1}\vvv{0}, \ggg{1}\bbb{1}\vvv{1} \} \\
            \wordSet[4]{S} &= \{ \bbb{0}\vvv{0}\rrr{0}, \bbb{0}\vvv{1}\rrr{0}, \bbb{1}\vvv{0}\rrr{0}, \bbb{1}\vvv{1}\rrr{1}, \bbb{0}\vvv{0}\rrr{1}, \bbb{0}\vvv{1}\rrr{1}, \bbb{1}\vvv{0}\rrr{1}, \bbb{1}\vvv{1}\rrr{0} \} \\
            \wordSet[5]{S} &= \{ \vvv{0}\rrr{0}\ooo{0}, \vvv{1}\rrr{0}\ooo{1}, \vvv{0}\rrr{0}\ooo{1}, \vvv{1}\rrr{1}\ooo{0}, \vvv{0}\rrr{1}\ooo{0}, \vvv{1}\rrr{1}\ooo{1}, \vvv{0}\rrr{1}\ooo{1}, \vvv{1}\rrr{0}\ooo{0} \}
        \end{aligned}$
        \caption{Multiversal cycle with $r = kn = 2 \cdot 3 = 6$ and
        $S = 000 \cdot 000 \cdot 001 \cdot 001 \cdots 111 \cdot 111$ by Theorem~\ref{thm:multiversalBinary}.
        }
        \label{fig:multiversalBinary3_r6}
    \end{subfigure}
    \caption{Multiversal cycles of $\mathcal{W} = \BINARY{n}$ for $n=3$ with multiplicity (or repetition) $r = kn$ for (a) $k=1$ and (b) $k=2$.
    The bits cycle through $r$ different colours and the words starting at a particular colour form $\BINARY{n}$.
    In particular, $\wordSet[0]{S}$ has the words starting with a \rrr{red} bit.
    Similarly, the individual word $000$ appears $r$ times, and each time it starts with a different colour.
    }
    \label{fig:multiversalBinary3}
\end{figure}

\subsection{Multiversal Cycles}
\label{sec:binary_multiversal}

Multiversal cycles with multiplicity $r=n$ first appeared in the context of shorthand universal cycles of $\PERMS{n}$ \cite{ruskey2010explicit}.
Here we define them for arbitrary sets of words and multiplicities.

\begin{definition} \label{def:multiversal}
Let $\mathcal{W}$ be a set of words of length $n$, and $l = |\mathcal{W}| - 1$. 
A sequence $S = s_0 s_1 \cdots s_{|S|-1}$ is called a \emph{multiversal cycle of $\mathcal{W}$ of multiplicity (or repetition) $r$} if $|S| = r \cdot |\mathcal{W}|$ and for all $0 \leq j < r$ we have
\begin{equation*}
\wordSet[n,j]{S} \walrus 
\begin{alignedat}[t]{6}
\{
    & s_j && \, s_{j+1} && {} \mathbin{\cdots} {}  && s_{j+n-1}, \\
    & s_{j+r} && \, s_{j+r+1} && \cdots  && s_{j+r+n-1}, \\ 
    & s_{j+2r} && \, s_{j+2r+1} && \cdots && s_{j+2r+n-1}, \\
    & \ldots, \\
    & s_{j+lr} && \, s_{j+lr+1} && \cdots && s_{j+lr+n-1} \} = \mathcal{W}.
\end{alignedat}
\end{equation*}
\end{definition}

Note that Definition \ref{def:multiversal} requires every word to appear exactly once at every index congruence class modulo its multiplicity (or repetition) $r$.
This can be seen by the colours in Figure \ref{fig:multiversalBinary3}.










\subsection{Lexicographic Order of Binary Words}
\label{sec:binary_lemmas}

Here we present two lemmas related to lexicographic order, or more specifically, lexicographic order with each word repeated some number of times.
Lemma \ref{lem:binarySame} considers subwords between two copies of the same word, while Lemma \ref{lem:binaryNext} considers subwords between consecutive words in lexicographic order.
Note that we view lexicographic order to be cyclic in the sense that the last word is followed by the first word.
A more general version of Lemma \ref{lem:binaryNext} was previously shown in \cite[Lemma 4]{alvarez2016perfect}, but we include our own proof as it provides a more explicit construction that forms a basis for multiversal cycles of unlabeled binary words in Section \ref{sec:unlabeled}.

\begin{lemma} \label{lem:binarySame}
Let $w = w_0 w_1 \cdots w_{n-1} \in \BINARY{n}$ and $i \in [n]$.
There exists $v = v_0 v_1 \cdots v_{n-1} \in \BINARY{n}$ such that $w = v_i v_{i+1} \cdots v_{n-1} v_0 v_1 \cdots v_{i-1}$.
\end{lemma}
This means that every binary word $w \in \BINARY{n}$ is the subword starting at index $i$ within $vv$ for some $v \in \BINARY{n}$.
Alternatively, every binary word $w$ is the $i$th rotation of some word $v$.
\begin{proof}
Use $v = w_{n-i} w_{n-i+1} \cdots w_{n-1} w_0 w_1 \cdots w_{n-i-1} \in \BINARY{n}$.\!\!\!\! 
\end{proof}

\begin{lemma} \label{lem:binaryNext}
Let $w = w_0 w_1 \cdots w_{n-1} \in \BINARY{n}$ and $i \in [n]$.
There exist consecutive $u = u_0 u_1 \cdots u_{n-1}$ and $v = v_0 v_1 \cdots v_{n-1}$ in the lexicographic order of $\BINARY{n}$ such that $w = u_i u_{i+1} \cdots u_{n-1} v_0 v_1 \cdots v_{i-1}$.
\end{lemma}
This means that every binary word is the subword starting at index $i$ within $uv$ for $u,v \in \BINARY{n}$ that are consecutive in lexicographic order.
\begin{proof} 
Let $\rrr{w} = \rrr{w_0 w_1 \cdots w_{n-1}} \in \BINARY{n}$ and $i \in [n]$ be arbitrary. 
We set the suffix of $\ggg{u}$ and prefix of $\bbb{v}$ as follows.
\begin{tightitemize}
\item $\ggg{u_i u_{i+1} \cdots u_{n-1}} = \rrr{w_0 w_1 \cdots w_{n-i-1}}$,
\item $\bbb{v_0 v_1 \cdots v_{i-1}} = \rrr{w_{n-i} w_{n-i+1} \cdots w_{n-1}}$.
\end{tightitemize}
This ensures that $\rrr{w} = \rrr{w_0 w_1 \cdots w_{n-1}}$ is the substring $\ggg{u_i u_{i+1} \cdots u_{n-1}} \bbb{v_0 v_1 \cdots v_{i-1}}$ within $uv$. 
It remains to set the unspecified prefix of $u$ and suffix of $v$ to ensure they are consecutive in cyclic lexicographic order. 
For convenience we write $\rrr{w} = \rrr{\alpha \beta}$ where $\rrr{\alpha} = \rrr{w_0 w_1 \cdots w_{n-i-1}}$ and $\rrr{\beta} = \rrr{w_{n-i} w_{n-i+1} \cdots w_{n-1}}$, with $a \walrus |\alpha| = n-i$ and $b \walrus |\beta| = i$.
We proceed with cases.
\smallskip

\proofcase{1}
$\rrr{\alpha} = \rrr{1^a}$ and $\rrr{\beta} = \rrr{0^b}$.\\
In this case we set the prefix of $\ggg{u}$ and suffix of $\bbb{v}$ as 
\begin{tightitemize}
\item $\ggg{u_0 u_1 \cdots u_{i-1}} = \ggg{1^{i}}$, 
\item $\bbb{v_i v_{i+1} \cdots v_{n-1}} = \bbb{0^{n-i}}$.
\end{tightitemize}
This ensures that $\ggg{u} = \ggg{1^n}$ is the last word in lexicographic order and $\bbb{v} = \bbb{0^n}$ is the first word in lexicographic order as visualized below.

\smallskip
$\begin{array}{@{}r@{\,}*{8}{@{\,}c@{\,}}}
\ggg{u} =& \ggg{1} & \ggg{1} & \ggg{\cdots} & \ggg{1} & \rrr{1} & \rrr{1} & \rrr{\cdots} & \rrr{1} \\
\bbb{v} =& \rrr{0} & \rrr{0} & \rrr{\cdots} & \rrr{0} & \bbb{0} & \bbb{0} & \bbb{\cdots} & \bbb{0}
\end{array}$

\smallskip
\proofcase{2}
$\rrr{\alpha} = \rrr{1^a}$ and $\rrr{\beta} \neq 0^b$. \\
Let $j \geq n-i $ be the maximum index with $\rrr{w_j} = \rrr{1}$.
This means $\rrr{\beta} = \rrr{w_{n-i} w_{n-i+1} \cdots w_{j-1} 1 0^{n-j-1}}$.
Then we define the prefix of $\ggg{u}$ and the suffix of $\bbb{v}$ as follows.
\begin{tightitemize}
\item $\ggg{u_0 u_1 \cdots u_{i-1}} = \ggg{w_{n-i} w_{n-i+1} \cdots w_{j-1} 0 1^{n-j-1}}$,
\item $\bbb{v_i v_{i+1} \cdots v_{n-1}} = \bbb{0^{a}} = \flip{\bbb{\alpha}}$.
\end{tightitemize}
This ensures that $\ggg{u}$ is followed by $\bbb{v}$ in lexicographic order as visualized below.

\smallskip
$\begin{array}{@{}r@{\,}*{13}{@{\,}c@{\,}}}
\ggg{u} =& \ggg{w_{n-i}} & \ggg{w_{n-i+1}} & \ggg{\cdots} & \ggg{w_{j-1}} & \ggg{0} & \ggg{1} & \ggg{1} & \ggg{\cdots} & \ggg{1} & \rrr{1} & \rrr{1} & \rrr{\cdots} & \rrr{1} \\
\bbb{v} =& \rrr{w_{n-i}} & \rrr{w_{n-i+1}} & \rrr{\cdots} & \rrr{w_{j-1}} & \rrr{1} & \rrr{0} & \rrr{0} & \rrr{\cdots} & \rrr{0} & \bbb{0} & \bbb{0} & \bbb{\cdots} & \bbb{0}
\end{array}$

\smallskip

\proofcase{3}
$\rrr{\alpha} \neq 1^a$.\\
Let $j < n-i$ be the maximum index with $\rrr{w_j}=\rrr{0}$.
That means $\rrr{\alpha} = \rrr{w_0 w_1 \cdots w_{j-1} 0 1^{a-j-1}}$.
Then we complete $\ggg{u}$ and~$\bbb{v}$ as follows.
\begin{tightitemize}
\item $\ggg{u_0 u_1 \cdots u_{i-1}} = \ggg{w_{n-i} w_{n-i+1} \cdots w_{n-1}} = \ggg{\beta}$.
\item $\bbb{v_i v_{i+1} \cdots v_{n-1}} = \bbb{w_0 w_1 \cdots w_{j-1} 1 0^{a-j-1}}$.
\end{tightitemize}
This ensures that $\ggg{u}$ is followed by $\bbb{v}$ in lexicographic order as visualized below.

\smallskip
$\begin{array}{@{}r@{\,}*{12}{@{\,}c@{\,}}}
\ggg{u} =& \ggg{w_{n-i}} & \ggg{w_{n-i+1}} & \ggg{\cdots} & \ggg{w_{n-1}} & \rrr{w_0} & \rrr{w_1} & \rrr{\cdots} & \rrr{w_{j-1}} & \rrr{0} & \rrr{1} & \rrr{\cdots} & \rrr{1} \\
\bbb{v} =& \rrr{w_{n-i}} & \rrr{w_{n-i+1}} & \rrr{\cdots} & \rrr{w_{n-1}} & \bbb{w_0} & \bbb{w_1} & \bbb{\cdots} & \bbb{w_{j-1}} & \bbb{1} & \bbb{0} & \bbb{\cdots} & \bbb{0}
\end{array}$ \qedhere
\end{proof}

\subsection{Construction via Lexicographic Order}
\label{sec:binary_construction}

Now we prove that multiversal cycles of $\BINARY{n}$ with multiplicity $r=kn$ can be constructed by concatenating (multiple) copies of each word in $\BINARY{n}$ in lexicographic order.
This result was previously proven for $k=1$ (or equivalently, $r=n$) in \cite{alvarez2016perfect} and $k \geq 1$ in \cite{fillmore2026existence}.
Figure \ref{fig:multiversalBinary3_r3} and \ref{fig:multiversalBinary3_r6} illustrate the concatenation where each word is repeated $k=1$ and $k=2$ times, respectively.

\begin{theorem}[\cite{fillmore2026existence}, Proposition 4.1] \label{thm:multiversalBinary}
Let the lexicographic order of $\BINARY{n}$ be $\alpha_0, \alpha_1, \ldots, \alpha_{N-1}$ with $N = 2^n$.
The sequence $\alpha_0^k \alpha_1^k \cdots \alpha_{N-1}^k$ is a multiversal cycle of $\BINARY{n}$ with multiplicity $r=kn$ for any $k \geq 0$. 
\end{theorem}
\begin{proof}
Let $k \geq 0$ be arbitrary and $S \walrus \alpha_0^k \alpha_1^k \cdots \alpha_{N-1}^k$.
We have that $|S| = k \cdot n \cdot |\BINARY{n}|$, so it remains to show that $\wordSet[n, j]{S} = \BINARY{n}$ for all $0 \leq j < r = nk$.

If $0 \leq j \leq (k-1) \cdot n$ then the elements of $\wordSet[n,j]{S}$ are exactly the subwords starting at position $j$ within $\alpha_{t} \alpha_{t}$ for $t \in [N]$. So the result follows by Lemma \ref{lem:binarySame} for $i = j \pmod n$.

If $(k-1)\cdot n < j < kn$ then $\wordSet[n,j]{S}$ is the set of substrings starting at position $j$ within $\alpha_t \alpha_{t+1}$ for $t \in [N]$. So Lemma \ref{lem:binaryNext} for $i = j \pmod n$ gives the desired result.
\end{proof}

It is worth noting that orders other than lexicographic order may not work with Theorem~\ref{thm:multiversalBinary}.
For example, consider the binary reflected Gray code $00, 01, 11, 10$ for $\BINARY{n}$ with $n=2$ \cite{mutze2023combinatorial}.
Concatenating these words gives $S = s_0 s_1 \cdots s_7 = \rrr{0}\bbb{0}\rrr{0}\bbb{1}\rrr{1}\bbb{1}\rrr{1}\bbb{0}$ which is not a multiversal cycle with multiplicity $m=n=2$.
In particular, $\wordSet[2,1]{S} = \{\bbb{0}\rrr{0}, \bbb{1}\rrr{1}, \bbb{1}\rrr{1}, \bbb{0}\rrr{0}\} \neq \BINARY{2}$.

\section{Unlabeled Binary Multiversal Cycles}
\label{sec:unlabeled}

In this section we construct multiversal cycles of unlabeled binary words.
This means that every binary word $w$ is viewed as being equivalent to its complement $\flip{w}$.
Somewhat informally we let $\UNLABELED{n}$ denote any set containing exactly one representative of each of the $2^{n-1}$ such pairs.
In other words, $\UNLABELED{n}$ is any maximal complement-free subset of $\BINARY{n}$.
One particular representation of $\UNLABELED{n}$ is the set of binary words whose first bit is $0$ which we denote by $\BINARYzero{n} \walrus \{b_0 \cdots b_{n-1} \in \BINARY{n} \colon b_0 = 0\}$.
For example, $\UNLABELED{2}$ can be represented as 
$\{01, 11\}$ or $\{10,11\}$ or $\BINARY[0]{2} = \{00,01\}$.
Examples of multiversal cycles of unlabeled binary words for $n=3$ with multiplicity $r=3$ and $r=6$ appear in Figure \ref{fig:multiversalUnlabeled3}.
Lemmas analogous to those in Section \ref{sec:binary_lemmas} are proven in \ref{sec:unlabeled_lemmas}, and then our construction appears in Section \ref{sec:unlabeled_construction}.

\begin{figure}
    \begin{subfigure}{0.445\textwidth}
        \raggedright
        $\begin{aligned}
            S &= \rrr{0}\ggg{0}\bbb{0}\rrr{0}\ggg{0}\bbb{1}\rrr{0}\ggg{1}\bbb{0}\rrr{0}\ggg{1}\bbb{1} \\
            \wordSet[0]{S} &= \{ \rrr{0}\ggg{0}\bbb{0}, \rrr{0}\ggg{0}\bbb{1}, \rrr{0}\ggg{1}\bbb{0}, \rrr{0}\ggg{1}\bbb{1} \} = \UNLABELED{3} \\
            \wordSet[1]{S} &= \{ \ggg{0}\bbb{0}\rrr{0}, \ggg{0}\bbb{1}\rrr{0}, \ggg{1}\bbb{0}\rrr{0}, \ggg{1}\bbb{1}\rrr{0} \} = \UNLABELED{3} \\
            \wordSet[2]{S} &= \{ \bbb{0}\rrr{0}\ggg{0}, \bbb{1}\rrr{0}\ggg{1}, \bbb{0}\rrr{0}\ggg{1}, \bbb{1}\rrr{0}\ggg{0} \} = \UNLABELED{3} \\
            \\
            \\
            \\
        \end{aligned}$
        \caption{Multiversal cycle with $r=kn=1\cdot3=3$ and $S = 000 \cdot 001 \cdot 010 \cdot 011$ via Theorem~\ref{thm:multiversalUnlabeled}.}
        \label{fig:multiversalUnlabeled3_r3}
    \end{subfigure}
    \hfill
    \begin{subfigure}{0.54\textwidth}
        \raggedright
        $\begin{aligned}
            S &= \rrr{0}\ooo{0}\yyy{0}\ggg{0}\bbb{0}\vvv{0}\rrr{0}\ooo{0}\yyy{1}\ggg{0}\bbb{0}\vvv{1}\rrr{0}\ooo{1}\yyy{0}\ggg{0}\bbb{1}\vvv{0}\rrr{0}\ooo{1}\yyy{1}\ggg{0}\bbb{1}\vvv{1} \\
            \wordSet[0]{S} &= \{ \rrr{0}\ooo{0}\yyy{0}, \rrr{0}\ooo{0}\yyy{1}, \rrr{0}\ooo{1}\yyy{0}, \rrr{0}\ooo{1}\yyy{1} \} = \UNLABELED{3} \\
            \wordSet[1]{S} &= \{ \ooo{0}\yyy{0}\ggg{0}, \ooo{0}\yyy{1}\ggg{0}, \ooo{1}\yyy{0}\ggg{0}, \ooo{1}\yyy{1}\ggg{0} \} = \UNLABELED{3} \\
            \wordSet[2]{S} &= \{ \yyy{0}\ggg{0}\bbb{0}, \yyy{1}\ggg{0}\bbb{0}, \yyy{0}\ggg{0}\bbb{1}, \yyy{1}\ggg{0}\bbb{1} \} = \UNLABELED{3} \\
            \wordSet[3]{S} &= \{ \ggg{0}\bbb{0}\vvv{0}, \ggg{0}\bbb{0}\vvv{1}, \ggg{0}\bbb{1}\vvv{0}, \ggg{0}\bbb{1}\vvv{1} \} = \UNLABELED{3} \\
            \wordSet[4]{S} &= \{ \bbb{0}\vvv{0}\rrr{0}, \bbb{0}\vvv{1}\rrr{0}, \bbb{1}\vvv{0}\rrr{0}, \bbb{1}\vvv{1}\rrr{0} \} = \UNLABELED{3} \\
            \wordSet[5]{S} &= \{ \vvv{0}\rrr{0}\ooo{0}, \vvv{1}\rrr{0}\ooo{1}, \vvv{0}\rrr{0}\ooo{1}, \vvv{1}\rrr{0}\ooo{0} \} = \UNLABELED{3}
        \end{aligned}$
        \caption{Multiversal cycle with $r=kn=2\cdot3=6$ and $S = 000 \cdot 000 \cdot 001 \cdot 001 \cdot 010 \cdot 010 \cdot 011 \cdot 011$ via Theorem~\ref{thm:multiversalUnlabeled}.}
        \label{fig:multiversalUnlabeled3_r6}
    \end{subfigure}
    \caption{Multiversal cycles of $\UNLABELED{n}$ for $n=3$ with multiplicity $r$.
    Here the sets are identical in an unlabeled sense.
    For example, in (a) the pair $011 = \overline{100}$ appears as $\rrr{0}\ggg{1}\bbb{1} \in \wordSet[0]{S}$ and $\ggg{1}\bbb{0}\rrr{0} \in \wordSet[1]{S}$.
    }
    \label{fig:multiversalUnlabeled3}
\end{figure}

\subsection[Lexicographic Order of Binary Words starting with 0]{Lexicographic Order of Binary Words starting with $0$}
\label{sec:unlabeled_lemmas}

In the following, we present two lemmas analogous to Lemma \ref{lem:binarySame} and \ref{lem:binaryNext}, but for the set $\BINARYzero{n}$.
As before, we consider the lexicographic order on $\BINARYzero{n}$ to be cyclic, that means the last word $01^{n-1}$ is followed by the first word $0^n$ of $\BINARYzero{n}$.
To ease notation, for binary words $w, v \in \BINARY{n}$ we write $w \equiv v$ to denote that either $w = v$ or $w = \flip{v}$.

\begin{lemma} \label{lem:unlabeledSame}
    Let $i \in [n]$ be arbitrary. For every $w \in \BINARY{n}$ there exists some $v = v_0 v_1 \cdots v_{n-1} \in \BINARYzero{n}$ such that $w \equiv v_i \cdots v_{n-1} v_0 v_1 \cdots v_{i-1}$. 
\end{lemma}
This means that for every binary word $w$ either $w$ itself or its complement $\flip{w}$ is the subword starting at index $i$ within $v v$ for some $v \in \BINARYzero{n}$.
\begin{proof}
    Let $i \in [n]$ and $w \in \BINARY{n}$ be arbitrary. 
	If $w_{n-i} = 0$, then $w$ itself can be found as a subword of  $v v$ starting at position $i$ for $v = w_{n-i} w_{n-i + 1} \cdots w_{n-1} w_0 w_1 \cdots w_{n-i-1}$.
	If $w_{n-i} = 1$, then $w$ itself cannot be found as a subword starting at position $i$ in any $vv$, since every $v\in \BINARYzero{n}$ has $0$ as the first bit.
    But for $v = \flip{w}_{n-i} \flip{w}_{n-i + 1} \cdots \flip{w}_{n-1} \flip{w}_0 \flip{w}_1 \cdots \flip{w}_{i-1}$, the complement $\flip{w}$ is a subword of $vv$ starting at position $i$, so $w \equiv v_i \cdots v_{n-1} v_0 v_1 \cdots v_{i-1}$. 
\end{proof}

\begin{lemma}\label{lem:unlabeledNext}
    Let $i \in [n]$ be arbitrary. 
	For every $w \in \BINARY{n}$ there exist consecutive \mbox{$u = u_0 u_1 \cdots u_{n-1}$} and $v = v_0 v_1 \cdots v_{n-1}$ in the lexicographic order of $\BINARYzero{n}$ such that \mbox{$w \equiv u_i u_{i+1} \cdots u_{n-1} v_{0} v_1 \cdots v_{i-1}$}. 
\end{lemma}
That means for every binary word $w$ either $w$ itself or its complement $\flip{w}$ is the subword starting at index $i$ within $u v$ for some $u, v \in \BINARYzero{n}$ that are consecutive in the cyclic lexicographic order.
\begin{proof}
    Let $i \in [n]$ and $w \in \BINARY{n}$ be arbitrary. 
	We distinguish 3 cases.

    \medskip
	
	\proofcase{1} $\rrr{w} = \rrr{1^{n-i}0 ^{i}}$ or $\rrr{w} = \rrr{0^{n-i}1^i}$.\\
	Then $\ggg{u} = \ggg{01^{n-1}}$ and $\bbb{v} = \bbb{0^{n}}$ satisfy that $\rrr{w} \equiv \ggg{u_{i} u_{i+1} \cdots u_{n-1}} \bbb{v_0 v_1 \cdots v_{i-1}}$ as illustrated below. 

    \smallskip
	$\begin{array}{@{}r@{\,}*{7}{@{\,}c@{\,}}}
		\ggg{u} =& \ggg{0} & \ggg{1} & \ggg{\cdots} & \ggg{1} & \rrr{1} & \rrr{\cdots} & \rrr{1} \\
		\bbb{v} =& \rrr{0} & \rrr{0} & \rrr{\cdots} & \rrr{0} & \bbb{0} & \bbb{\cdots} & \bbb{0}
	\end{array}$

    \smallskip
    \noindent Note that we consider the lexicographic order on $\BINARYzero{n}$ to be cyclic, so $u$ and $v$ are consecutive. 

    \smallskip
	For the remaining cases we assume $w\in \BINARY{n} \setminus \{1^{n-i} 0 ^{i},  0^{n-i} 1^i\}$ without loss of generality.

    \medskip
	
	\proofcase{2} $\rrr{w_{n-i}} = \rrr{0}$.\\
	By Lemma \ref{lem:binaryNext} for $i$ and $w$ there exist $\ggg{u}, \bbb{v} \in \BINARY{n}$ such that $\rrr{w} = \ggg{u_i \cdots u_{n-1}} \bbb{v_0 \cdots v_{i-1}}$ and $\ggg{u}, \bbb{v}$ are consecutive in the lexicographic order of $\BINARY{n}$. 
	From $\bbb{v_0} = \rrr{w_{n-i}} = 0$ it follows that $\bbb{v} \in \BINARYzero{n}$. 
	Since $\ggg{u}$ is lexicographically smaller than $\bbb{v}$, and $w \neq 1^{n-i}0^i$, it follows that $\ggg{u_0} = \ggg{0}$ and hence $\ggg{u} \in \BINARYzero{n}$. 
	So $\ggg{u}$ and $\bbb{v}$ satisfy the required properties.  

    \medskip
	 
	\proofcase{3} $\rrr{w_{n-i}} = \rrr{1}$.\\
	Here we work with the complement of $\rrr{w}$. 
	By Lemma \ref{lem:binaryNext} for $i$ and $\flip{w}$ there exist $\ggg{u}, \bbb{v} \in \BINARY{n}$ such that $\rrr{\flip{w}} = \ggg{u_i \cdots u_{n-1}} \bbb{v_0 \cdots v_{i-1}}$ and $\ggg{u}, \bbb{v}$ are consecutive in the lexicographic order of $\BINARY{n}$. 
	Since $\rrr{w_{n-i}} = \rrr{1}$ we have $\bbb{v_0} =\rrr{\flip{w}_{n-i}} = 0$. 
	So $\bbb{v} \in \BINARYzero{n}$, and because $\ggg{u}$ is lexicographically smaller than $\bbb{v}$, and $w \neq 0^{n-i}1^i$, we also have $\ggg{u} \in \BINARYzero{n}$. 
	Then $\ggg{u}$ and $\bbb{v}$ are two consecutive words from $\BINARYzero{n}$ such that $ \rrr{w} \equiv \ggg{u_i \cdots u_{n-1}} \bbb{v_0 \cdots v_{i-1}}$. 
    \end{proof}

\subsection{Construction via Lexicographic Order}
\label{sec:unlabeled_construction}
We show that multiversal cycles of $\UNLABELED{n}$ can be constructed by concatenating several copies of each word from $\BINARYzero{n}$ in lexicographic order. Theorem \ref{thm:multiversalUnlabeled} below follows from Lemmas \ref{lem:unlabeledSame} and \ref{lem:unlabeledNext} with an analogous proof to that of Theorem \ref{thm:multiversalBinary} for multiversal cycles of $\BINARY{n}$.

\begin{theorem} \label{thm:multiversalUnlabeled}
	Let the lexicographic order of $\BINARYzero{n}$ be $\alpha_0, \alpha_1, \ldots, \alpha_{N-1}$ with $N = 2^{n-1}$.
    For any $k \geq 0$ the sequence $\alpha_0^k \alpha_1^k \cdots \alpha_{N-1}^k$ is a multiversal cycle of $\UNLABELED{n}$ with multiplicity~$r=kn$.
\end{theorem}

\section{Shorthand Twori for Permutations}
\label{sec:tworiPerms}

In this section we combine our previous results to show that shorthand universal twori exist for all odd $n$ by giving an explicit construction. 
We restate our main Theorem below. 

\tworusExists*

We first define a general construction of a two-row torus from a given list of edges and binary sequence. Then we show that with the right choice of edges and binary sequence, the construction gives a shorthand universal tworus for permutations.  

\begin{definition} \label{def:torus}

    We construct a torus $E \torusOp S$ from a list of edges $E$ and a binary sequence $S$ in the following two steps. 
    \begin{enumerate}
        \item For a single edge $e = \{u, v\}$ of $\Kn{n}$, where $u < v$, and a bit $w \in \{0, 1\}$ we define the corresponding column $\pairs{\pair{a}{b}}$ by setting $a = u$, $b = v$ if $w = 0$, and $a = v$, $b = u$ if $w = 1$. That means we have
        \begin{align*}
            \pairs{\pair{a}{b}} \walrus
                \begin{cases}
                    \pairs{\pair{u}{v}}, \text{ if } w = 0\\
                    \pairs{\pair{v}{u}}, \text{ if } w = 1.
                \end{cases}
        \end{align*}
        \item For a list $E = [e_0, \ldots, e_{l-1}]$ of not necessarily unique edges of $\Kn{n}$ and a binary sequence $S = w_0 w_1 \cdots w_{l-1}$ of length $l$ we define a $2$-by-$l$ torus
        \begin{align*}
            E \torusOp S \walrus \pairs{\pair{a_0}{b_0} \pair{a_1}{b_1} \cdots \pair{a_{l-1}}{b_{l-1}}}.
        \end{align*}
        In other words, each edge contributes a column, with its smaller or larger vertex in the top row depending on the corresponding bit in the binary sequence. 
    \end{enumerate}
\end{definition}


We create our shorthand universal torus of $\PERMS{n}$ by concatenating several copies of a universal cycle of maximum matchings of $\Kn{n}$.
The corresponding columns of the tworus are ordered according to an appropriate binary sequence as described above, such that the blocks in the tworus corresponding to each copy of the universal cycle all match, but no two blocks are identical. An example is given in Figure \ref{fig:blocks}.

\begin{lemma} \label{lem:tworus}
    Let $n = 2m+1 \geq 3$ be odd. Given a universal cycle $C$ of $\MATCHINGS[m]{\Kn{n}}$ and a multiversal cycle $S$ of $\UNLABELED{m}$ of multiplicity $\frac{n!}{2^m}$, the torus $C^{2^{m-1}} \torusOp S$ is a $2$-by-$\frac{n!}{2}$ shorthand universal torus of $\PERMS{n}$ with $2$-by-$m$ windows.
\end{lemma}

\begin{proof}
    Let $n = 2m+1 \geq 3$ be an arbitrary odd integer. To ease notation we use $E = C^{2^{m-1}}$ to denote the sequence of edges consisting of $2^{m-1}$ repetitions of the universal cycle $C$ of $\MATCHINGS[m]{\Kn{n}}$.
    We construct the torus $E \torusOp S$ as described in Definition \ref{def:torus}. 
    Note that $|S| = \frac{n!}{2^m} \cdot |\BINARYzero{m}| = \frac{n!}{2}$ and by Remark \ref{rem:div} we also have $|E| = |C| \cdot 2^{m-1} = \frac{n!}{2}$.

    In a torus of dimension $2$-by-$\frac{n!}{2}$ there are precisely $n! = |\PERMS{n}|$
    distinct windows. It therefore suffices to show that every shorthand permutation from
    $\PERMS[n-1]{n}$ can be found in some window of the torus to see that every
    permutation appears exactly once. 

    Let $\pi = p_0 p_1 \cdots p_{2m-1} \in \PERMS[n-1]{n}$ be a shorthand permutation. 
    Consider the
    ordered matching $\pairs{\pair{p_0}{p_m}, \pair{p_1}{p_{m+1}}, \ldots,
    \pair{p_{m-1}}{p_{2m-1}}} \in \MATCHINGS[m]{\Kn{n}}$.
    The universal cycle
    $C$ contains this ordered matching starting at some index $t$, and thus the list of edges $E = C^{2^{m-1}}$ contains this matching at the indices $t + k \cdot \frac{n!}{2^m}$ for all $k \in [2^{m-1}]$. 
    Let $w = w_0 w_1 \cdots w_{m-1}$ be the binary sequence with $w_i = 0$ if $p_i < p_{i+m}$ and $w_i = 1$ if $p_i > p_{i+m}$. 
    Since $S$ is a multiversal cycle of $\UNLABELED{m}$ we know that $\wordSet[m,t]{S}$ is a maximal complement-free subset of $\BINARY{m}$.
    So there exists $k^* \in [2^{m-1}]$ for which $S$ contains either $w$ or $\flip{w}$ as a subword starting at index $s \walrus t+ k^* \cdot \frac{n!}{2^m}$.

    If $S$ contains $w$ itself as a subword, then by construction the $2$-by-$m$ block with entries $\pairs{\pair{p_0}{p_m} \pair{p_1}{p_{m+1}} \cdots
    \pair{p_{m-1}}{p_{2m-1}}}$ appears in the torus $E \torusOp S$. 
    Hence the corresponding $2$-by-$m$ window which has its top row in the first row of the torus and its leftmost column at position $t+ k^* \cdot \frac{n!}{2^m}$ contains the shorthand permutation $\pi$ in row-major order. 
    
    If $S$ contains $\flip{w}$ as a subword, then the torus $E \torusOp S$ contains the flipped block with entries $\pairs{\pair{p_m}{p_0} \pair{p_m+1}{p_{}} \cdots
    \pair{p_{2m-1}}{p_{m-1}}}$, since complementing a bit in $S$ corresponds to swapping the top and bottom entry of the corresponding column in the torus. 
    Then the $2$-by-$m$ window which has its leftmost column at position $t+ k^* \cdot \frac{n!}{2^m}$ and its first row in the bottom row of the torus, such that the second row wraps around to the top row of the torus, contains the shorthand permutation $\pi$ in row-major order. 
%
%
%
%
%
\end{proof}

Theorem \ref{thm:tworus} is a direct consequence of Lemma \ref{lem:tworus}.

\begin{proof}[Proof of Thm \ref{thm:tworus}]
    For $n = 1$ we simply get the trivial empty torus, so let $n = 2m+1  \geq 3$. 
    Theorem $\ref{thm:uniMatchings}$ guarantees the existence of a universal cycle $C$ of $\MATCHINGS[m]{\Kn{n}}$ and from Theorem \ref{thm:multiversalUnlabeled} for $k = \frac{n!}{2^m m}$ we get a multiversal cycle $S$ of $\UNLABELED{m}$ of multiplicity $\frac{n!}{2^m}$. Then by Lemma \ref{lem:tworus} the torus $C^{2^{m-1}} \torusOp S$ is a $2$-by-$\frac{n!}{2}$ shorthand universal torus of $\PERMS{n}$ with $2$-by-$m$ windows.
\end{proof}


\section{Algorithmic Results}
\label{sec:alg}

Now we present our main algorithmic results. 
First, we consider the construction of multiversal cycles. 
Since the words in $\BINARY{n}$ and $\BINARY[0]{n}$ can be generated in lexicographic order in amortized $\bigO{1}$-time per word \cite{ruskey2003combinatorial} we can generate our multiversal cycles efficiently.

\begin{theorem}
    \label{th:algo_binary_cycle}
    There is an algorithm for generating multiversal cycles for $\BINARY{m}$ of
    multiplicity $r=km$ that takes amortized $\bigO{1}$-time to generate each bit.
\end{theorem}
\begin{proof}
    Algorithm \ref{alg:generate_multiversal_cycle} takes $\mathcal O(r \cdot 2^m)$ time to 
    generate a sequence of $\mathcal O(r \cdot 2^m)$ bits. Incrementing the bitstring leads to lexicographic order, so correctness follows from Theorem~\ref{thm:multiversalBinary}.
\end{proof}

\begin{theorem}
    \label{th:algo_unlabeled_cycle}
    There is an algorithm for generating multiversal cycles for $\UNLABELED{m}$ of
    multiplicity $r=km$ that takes amortized $\bigO{1}$-time to generate each bit.
\end{theorem}
\begin{proof}
    Analogous to Theorem \ref{th:algo_binary_cycle}, with correctness following
    from Theorem \ref{thm:multiversalUnlabeled}.
\end{proof}
\begin{algorithm}
    \SetKwFunction{GenerateMC}{GenerateMultiversalCycle}

    \Fn(){\GenerateMC{$m$, $r$}}{
        $k \gets r / m$\;
        $b = b_0\cdots b_{m - 1} \gets $ bitstring of $m$ zeroes\;
        \RepeatTimes{$2^m$} {
            \ForEach{$i \in [k \cdot m]$} {
                \Yield $b_{i \bmod m}$\;
            }
            increment $b$\;
        }
    }

    \caption{A remarkably simple algorithm to generate a multiversal cycle of $\BINARY{m}$ with multiplicity $r=km$.
    The increment step can be implemented in amortized $\bigO{1}$-time \cite{ruskey2003combinatorial} so each bit is generated in amortized $\bigO{1}$-time.
    A multiversal cycle of $\UNLABELED{m}$ with multiplicity $r=km$ is generated instead if the main block is repeated only $2^{m - 1}$ times.}
    \label{alg:generate_multiversal_cycle}
\end{algorithm}

\smallskip

Our construction of a tworus requires the construction of a universal cycle of ordered matchings $\MATCHINGS[m]{\Kn{n}}$.
This can be done by computing an Eulerian circuit in the associated transition graph.
However, once we have a shorthand universal cycle of $\MATCHINGS[m]{\Kn{n}}$, we can reuse it an exponential number of times to create a shorthand universal tworus for $\PERMS{n}$.
This leads to the following theorem.


\begin{theorem}
    \label{th:algo_tworus_known_cycle}
    Let $n = 2m + 1$. Given a universal cycle of the maximum ordered matchings of
    $\Kn{n}$, a $2$-by-$\frac{n!}{2}$ shorthand universal tworus for $\PERMS{n}$ can be
    generated in amortized $\bigO{1}$-time per column using additional space
    polynomial in $n$.
\end{theorem}

\begin{proof}
    Algorithm \ref{alg:generate_tworus} implements the tworus construction.  It
    implicitly constructs a multiversal cycle $S$ of $\UNLABELED{m}$, and given
    a universal cycle $M$ of ordered matchings of the $\Kn{n}$, it computes $M
    \torusOp S$ as defined in Definition \ref{def:torus}. The correctness
    follows from Lemma \ref{lem:tworus}.

    We compute the words of $\BINARYzero{m}$ in lexicographic order as in
    Algorithm \ref{alg:generate_multiversal_cycle}, to ensure that each
    successive bit can be computed in amortized constant time and that not all
    elements of $\BINARYzero{m}$ have to be stored in memory simultaneously.
    While $M$ has exponential size, this ensures that polynomial working
    space suffices.
    The inner loop is executed $|M| \cdot |\BINARYzero{m}| = \frac{n!}{2^m} \cdot 2^{m -
    1} = \frac{n!}{2}$ times and takes $\bigO{1}$ time for each iteration, which
    yields the runtime bound.
\end{proof}
\begin{algorithm}[htb]
    \SetKwFunction{GenerateTworus}{GenerateTworus} 
    \Fn(){\GenerateTworus{$n$, $M$}}{
        $m \gets \frac{n - 1}{2}$\;
        \tcp{Iterate over the unlabeled $m$-bit words in lexicographic order}
        $b = b_0 b_1 \cdots b_{m-1} \gets $ bitstring of $m$ zeroes \;
        \RepeatTimes{$2^{m - 1}$} {
            $i \gets 0$\;
            \tcp{Iterate over the edges in the universal cycle.}
            \ForEach{$e = \{u,v\} \in M$}{
                \eIf{$b_i = 0$}{
                    \Yield $\pair{u}{v}$
                 }{
                    \Yield $\pair{v}{u}$
                 }
                $i \gets (i+1) \bmod m$
            }
            increment $b$ \;
        }
    }
    \caption{Generate a tworus for $\PERMS{n}$ from $M$, a universal cycle of near perfect ordered matchings in the complete graph $\Kn{n}$ where $n=2m+1$.
    The tworus is a $2$-by-$\frac{n!}{2}$ shorthand universal torus with $2$-by-$m$ windows.
    Each column is yielded in amortized $\bigO{1}$-time.}
    \label{alg:generate_tworus}
\end{algorithm}



\section{Non-Existence Results}
\label{sec:nexists}

\newcommand{\symbolA}{\ensuremath{s}}
\newcommand{\symbolB}{\ensuremath{t}}
\newcommand{\columnA}{\ensuremath{A}}
\newcommand{\columnB}{\ensuremath{B}}

In Section \ref{sec:n5}, computer search determined that no shorthand universal torus exists
for $\PERMS{5}$ with 2-by-2 windows and torus sizes $3$-by-$40$, $5$-by-$24$, and $10$-by-$12$.
In this section, we give proofs for the nonexistence of the $3$-by-$40$ and $5$-by-$24$ tori without relying on exhaustive search.


\subsection{Torus with $\mathbf{R = 3}$}

\begin{lemma}
\label{lemma:t3_adj}
In a $3$-by-$40$ shorthand universal torus for $\PERMS{5}$ with $2$-by-$2$ windows, it is impossible for two adjacent columns to contain a common symbol, and it is impossible for a column to contain two copies of any symbol.
\end{lemma}

\begin{proof}
This follows directly from the fact that every $2$-by-$2$ window represents a permutation, and therefore may not contain any symbol twice.
\end{proof}

\tThreeNonexistence*

\begin{proof}
Any given symbol $\symbolA$ in the $3$-by-$40$ torus needs to occur $120/5 = 24$ times (by symmetry). By Lemma \ref{lemma:t3_adj}, $\symbolA$ can occur at most once in every second column, so at least $48$ columns are needed, but there are only $40$ columns.
\end{proof}

\subsection{Torus with $\mathbf{R = 5}$}

\begin{lemma}
\label{lemma:t5_notallpermutations}
In a $5$-by-$24$ shorthand universal torus for $\PERMS{5}$ with $2$-by-$2$
windows, it is impossible for every column to contain all five symbols exactly
once.
\end{lemma}

\begin{proof}
If each column contains all five symbols then neighboring columns will be cyclic shifts by either 2 or 3 positions. This is because in a torus with $R=5$ no symbol may repeat in any $2$-by-$2$ window. Hence, for some symbol $\symbolA$ in some column $\columnA$, only two possible positions are valid in an adjacent column $\columnB$. Choosing any of these positions will determine the entire column, as shown in Figure \ref{fig:cycles}.

As we only have two distinct types of adjacencies between neighboring columns, as soon as there are four columns a window has to repeat.
\end{proof}

\begin{figure}
\centering
\includegraphics[scale=0.6]{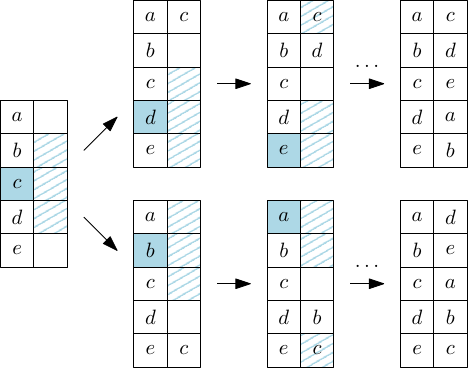}
\caption{Symbols determined in a $5$-by-$24$ shorthand universal torus for $\PERMS{5}$ with $2$-by-$2$ windows when an adjacent column contains all five symbols.}
\label{fig:cycles}
\end{figure}

\begin{lemma}
\label{lemma:t5_double}
In a $5$-by-$24$ shorthand universal torus for $\PERMS{5}$ with $2$-by-$2$
windows, if some column contains some symbol $\symbolA$ twice, then every second column
will contain two copies of $\symbolA$, and all the other columns will not contain any
copies of $\symbolA$.
\end{lemma}

\begin{proof}
Any column may contain at most two copies of $\symbolA$. In total, $120/5 = 24$ copies
of $\symbolA$ are needed, which implies that there must be a column containing two
copies of $\symbolA$ for every column containing $0$ copies of $\symbolA$.

Let $\columnA$ be a column that contains two copies of $\symbolA$ and note these copies cannot
be (cyclically) adjacent. Therefore, each cell of the columns adjacent to $\columnA$ shares a $2$-by-$2$ window with at least one copy of $\symbolA$, and hence cannot contain $\symbolA$.

For every maximal sequence of alternating columns with $2$ and $0$ copies of
$\symbolA$, respectively, either the sequence contains each type of column the same number of times and cyclically wraps around the torus or there is
one more column with $0$ copies of $\symbolA$ than there is with $2$ copies of $\symbolA$. 
In the first case, the torus contains no columns apart from those of the alternating sequence, and in the latter case, the torus cannot contain such a sequence as it would be of odd length.
\end{proof}

\begin{lemma}
\label{lemma:t5_twodoubles}
In a $5$-by-$24$ shorthand universal torus for $\PERMS{5}$ with $2$-by-$2$
windows, for two distinct symbols $ \symbolA,\symbolB$ both of the following situations are
impossible:
\begin{enumerate}
\item There is a column $\columnA$ that contains two copies of $\symbolA$ and two
	copies of $\symbolB$.
\item There are columns $\columnA, \columnB$ such that $\columnA$ contains two copies of $\symbolA$
	and $\columnB$ contains two copies of $\symbolB$.
\end{enumerate}
\end{lemma}

\NewDocumentCommand{\smallbmatrix}{m}{%
    \ensuremath{\left[\begin{smallmatrix}%
        #1%
    \end{smallmatrix}\right]}%
}

\begin{proof}
In the first case, by Lemma \ref{lemma:t5_double}, we know that every second column will contain two copies of $\symbolA$ and $\symbolB$ each, and all the other columns must not contain any $\symbolA$ or $\symbolB$. But then, no window of the type $\smallbmatrix{\symbolA & \symbolB \\ \phantom{x} & \phantom{x}}$ can exist.

In the second case, by Lemma \ref{lemma:t5_double}, we know that there is no column that contains copies of both $\symbolA$ and of $\symbolB$. But then, no window of the shape $\smallbmatrix{ \symbolA & \phantom{x} \\ \symbolB & \phantom{x} }$ can exist.
\end{proof}

\tFiveNonexistence*

\begin{proof}
By Lemma \ref{lemma:t5_notallpermutations} there is at least one column $\columnA$ that contains some symbol $\symbolA$ twice. For an adjacent column $\columnB$, by Lemma \ref{lemma:t5_double} no $\symbolA$ can occur in $\columnB$. Because $\columnB$ still contains five cells, some other symbol has to occur twice in $\columnB$. This contradicts Lemma \ref{lemma:t5_twodoubles}.
\end{proof}

\section{Final Remarks}
\label{sec:final}

In this paper we initiated the study of shorthand universal tori of permutations $\PERMS{n}$.
We used exhaustive computation to search for the existence and non-existence for $2$-by-$2$ windows tori with $n=5$.
Unlike more traditional de Bruijn tori for binary and $q$-ary words, our computational results did not immediately lead to an obvious conjecture for existence. 

We proved that universal tori of $\PERMS[n-1]{n}$ exist when $n=2m+1$ and the torus and windows have two rows (i.e., $R=r=2$).
Somewhat playfully, we refer to these tori as \emph{twori}.

Following a recent trend in Gray codes we showed that twori exist with a high degree of symmetry.
More specifically, there are twori that partition into $2^{m-1}$ matching blocks.
We describe the block structure as a universal cycle of ordered maximum matchings in the complete graph $\Kn{n}$.
In fact, we characterized when universal cycles of $\MATCHINGS[k]{\Kn{n}}$ exist.

The second ingredient in our tworus construction involves multiversal cycles.
We extended an existing result \cite{alvarez2016perfect, fillmore2026existence} by showing that lexicographic concatenations produce multiversal cycles of both binary words $\BINARY{m}$ and unlabeled words $\UNLABELED{m}$.
These concatenations can be generated very efficiently.
As a result, if we are given one block (i.e., a universal cycle of $\MATCHINGS[m]{\Kn{n}}$), then we can generate successive columns in a tworus in amortized $\bigO{1}$-time.

The most natural open problem is for which values $(R,C;r,c)$ do shorthand universal tori exist.
In future work we would like to construct a specific universal cycle for $\MATCHINGS[m]{\Kn{n}}$ that has an efficient successor rule \cite{gabric2019successor,sawada2026concatenation}. 
This will allow us to efficiently generate a specific tworus without the assumptions of Theorem \ref{th:algo_tworus_known_cycle}.
Having a specific tworus would also allow us to consider the position encoding/decoding problem \cite{shiu1997decoding,horan2016locating,etzion2024bruijn}.
Researchers can also consider tori beyond two dimensions~\cite{roig2021review}, multiset permutations rather than permutations \cite{sawada2021universal,sawada2023constructing}, or representations other than shorthand (e.g., order isomorphism \cite{johnson2009universal}).

\section*{Acknowledgements}

This paper originated at the workshop \emph{Combinatorics, Algorithm and Geometry}, which took place in Kassel in March 2026. 
Workshop funding from the DFG priority programme SPP 2458 \emph{Combinatorial Synergies} and the DFG Heisenberg grant 522790373 \emph{Combinatorial Algorithms} is gratefully acknowledged.
The authors wish to thank Stefan Felsner (TU Berlin) and Torsten M\"utze (University of Kassel) for productive discussions during the workshop.
The reviewers also helped correct several small but important errors and pointed us to the literature on perfect necklaces. 


\small
\bibliographystyle{abbrv}
\bibliography{refs}

\end{document}